\documentclass[reqno]{amsart}
\usepackage{pict2e}
\usepackage{amsmath,amssymb}
\usepackage{relsize}
\usepackage{graphicx,color}
\usepackage[all]{xy}

\theoremstyle{plain}
\newtheorem{introtheorem}{Theorem}

\newtheorem{theorem}{Theorem}[section]
\newtheorem{proposition}[theorem]{Proposition}
\newtheorem{lemma}[theorem]{Lemma}
\newtheorem{corollary}[theorem]{Corollary}

\theoremstyle{definition}
\newtheorem{definition}[theorem]{Definition}
\newtheorem{example}[theorem]{Example}

\theoremstyle{remark}
\newtheorem{remark}[theorem]{Remark}

\def\E{{\mathcal E}}
\def\B{{\mathcal B}}

\def\J{{\mathcal J}}

\def\K{{\mathcal H}}

\def\Z{{\mathbb Z}}

\def\H{{\mathcal H}}
\def\I{{\mathcal I}}
\def\T{{\mathcal T}}

\def\Hom{\mathrm{Hom}}

\def\cat0{\mathrm{cat}_0}

\def\ker{\mathrm{ker}}
\def\coker{\mathrm{coker}}
\def\im{\mathrm{im}}
\def\aut{\mathrm{aut}}
\def\H{\mathrm{Hom}}

\begin{document}
\title[]{On the group  of      self-homotopy equivalences of a   2-connected and 6-dimensional CW-complex}

\author{Mahmoud Benkhalifa}
\address{Department of Mathematics. College of  Sciences, University of Sharjah.  UAE. }

\email{mbenkhalifa@sharjah.ac.ae}


\begin{abstract}
Let  $X$ be  a   \text{\rm{2}}-connected and \text{\rm{6}}-dimensional CW-complex $X$ such that $H_{3}(X)\otimes\Z_2=0$. This paper aims  to describe  the  group  $\E(X)$ of    the  self-homotopy equivalences of  $X$ modulo   its normal subgroup $\E_{*}(X)$  of the elements that induce  the identity on the homology groups. Making use of the Whitehead exact sequence of $X$, denoted by WES$(X)$,  we    define the   group    $\Gamma\mathcal{S}(X)$ of  $\Gamma$-automorphisms of  WES$(X)$  and we prove that   $\E(X)/\E_*(X)\cong \Gamma\mathcal{S}(X).$
\end{abstract}
\maketitle
\section{Introduction}
Let  $\bf{CW^6_2}$ denote the category of  \text{\rm{2}}-connected and  \text{\rm{6}}-dimensional CW-complexes $X$ such that  such that $H_{3}(X)\otimes\Z_2=0$.  This paper aims  to describe and study the quotient group  $\E(X)/\E_{*}(X)$, where  $\E_{}(X)$ is the group  of  self-homotopy equivalences of an object  $X$ of $\bf{CW^6_2}$ and where   $\E_{*}(X)$  is its subgroup of the elements that induce  the identity on the homology groups $H_{*}(X)=H_{*}(X,\Z)$.  The group  $\E(X)$  and its subgroup  $\E_*(X)$ are  powerful algebraic invariant of $X$ but extremely difficult to compute in general (see for instance \cite{B1,B2,B3}).

For this purpose,    to any object $X$ of   $\bf{CW^6_2}$, we  assign the following exact sequence
$$H_{6}(X)\overset{b_{6}}{\longrightarrow  }H_{4}(X)\otimes\Z_{2}\oplus \K_{2}(H_{3}(X);\Z)\longrightarrow 
\pi_{5}(X)\overset{h_{5}}{\twoheadrightarrow  } H_{5}(X) $$

\noindent  called the Whitehead exact sequence of $X$ and denoted by WES$(X)$.   Here $h_{*}$ is the Hurewicz homomorphism and $\K_{2}(H_{3}(X);\Z)$ denotes the second homology group  of  the   group $H_{3}(X)$ with integer coefficients (see \cite{Bau2,B7,Whi} for more details).

It is worth mentioning that the Whitehead exact sequence can be defined for objects  in  algebraic categories such as differential graded Lie algebras or differential
graded free chain algebras (notice that we can define a homotopy theory for these categories such that they work
“similarly” as for CW-complexes). These are simpler cases and lead to more powerful theorems (see \cite{B4,B5,Ben3,B6,B7} for more details).

\smallskip

Inspiring by the ideas developed in \cite{B15,B7},  we introduce the group  $\Gamma\mathcal{S}(X)$   of the   $\Gamma$-automorphisms of  WES$(X)$  which are   graded automorphisms $f_{*}:H_{*}(X)\to H_{*}(X)$ satisfying a certain  algebraic condition (see definition \ref{d1}).  Then   we construct a map 
 	$$\Psi:\E(X)\to \Gamma\mathcal{S}(X)\,\,\,\,\,\,\,\,\,\,\,\,\,,\,\,\,\,\,\,\,\,\,\,\,\Psi([\alpha]) =H_{*}(\alpha).$$
The main theorem is as follows. 
\begin{introtheorem}
	Let $X$ be an   object  in  $\bf{CW^6_2}$.  There exists  a short exact sequence  of groups.
	$$\E_*(X) \rightarrowtail\E(X)\overset{\Psi}{\twoheadrightarrow    } \Gamma\mathcal{S}(X),$$
\end{introtheorem}
 The paper is organized as follows.

\noindent In Section 2, we recall the basic definitions of the 
Whitehead  exact sequence.    Section 3 is devoted to notion of the characteristic  extensions needed to define and study  the group  $\Gamma\mathcal{S}(X)$  of the $\Gamma$-automorphisms. In Section 4, we formulate
and prove the main theorem as well as  we give some  illustrations showing the efficiently of our results. 
\section{ The  Whitehead  exact sequence  of  an object in  $\bf{CW^6_2}$}
\subsection{The cellular complex and the Hurewicz homomorphism}
Let $X$ be a simply connected CW-complex defined by the collection
of its skeleta $(X^n)_{n\geq 0}$, where we can suppose
$X^0=X^1=\star$.

\noindent The long exact sequence of the pair $(X^n, X^{n-1})$ in
homotopy and in homology are connected by the Hurewicz homomorphism
$h_*$:
$$
\xymatrix{
	\cdots\ar[r]^{\hspace{-6mm} }&\pi_{n}(X^{n}) \ar[r]^{\hspace{-6mm} j_{n}} \ar[d]_{h_{n}}  & \pi_{n}(X^{n},X^{n-1})\ar[r]^{\beta_{n}}\ar[d]^{h_{n}}& \pi_{n-1}(X^{n-1}) \ar[d]^{h_{n-1}}\ar[r]&\cdots \\
	\cdots\ar[r]^{\hspace{-6mm} }&H_{n}(X^{n}) \ar[r]^{\hspace{-6mm} j^{H}_{n}} &
	H_{n}(X^{n},X^{n-1})\ar[r]^{\beta^{H}_{n}}& H_{n-1}(X^{n-1})\ar[r]&\cdots
}
$$
\begin{remark}
	The following elementary  facts are well-known.
	\begin{enumerate}
		\label{r5}
		\item $\pi_n(X^n, X^{n-1})$ is the free $\mathbb{Z}$-module generated by the $n$-cells of $X$.
		\item The homomorphism  $\beta_n : C_nX \rightarrow \pi_{n-1}(X^{n-1})$ represents  the attaching map for the $n$-cells $\vee S^n \rightarrow X^{n-1}$.
		\item $\ker \beta_{n}=\im j_n$
		\item $C_nX=\pi_n(X^n, X^{n-1})$ with the differential $d_n= j_{n-1}\circ\beta_{n}$ defines the cellular chain complex of $X$.  Its homology is of course the singular homology $H_*(X)$ and  we omit to refer to $\mathbb{Z}$, it is understood that we deal only with integral homology.
	\end{enumerate}
\end{remark}

\subsection{The definition of Whitehead's certain exact sequence}
Now Whitehead \cite{Whi} inserted the Hurewicz morphism in a long exact sequence connecting homology and homotopy. First he defined the following group
\begin{equation}\label{2}
	\Gamma_n(X)=\mbox{Im }( \pi_{n}(i_n) : \pi_n(X^{n-1}) \rightarrow
	\pi_n(X^n))=\mbox{ker }j_n\,\,\, ,\,\,\, \forall n\geq 2.
\end{equation}
where $i_n:X^{n-1} \hookrightarrow X^n$ is the inclusion. We notice that $\beta_{n+1}\circ d_{n+1}=0$ and so $\beta_{n+1}:
\pi_{n+1}(X^{n+1}, X^n)\rightarrow  \pi_n(X^n)$ factors through the
quotient: $b_{n+1}:  H_{n+1}(X)\rightarrow  \Gamma_n(X)$.

\noindent With this map, Whitehead  \cite{Whi} defined the following
sequence
\begin{equation}
	\label{3} \cdots \rightarrow
	H_{n+1}(X)\overset{b_{n+1}}{\longrightarrow }%
	\Gamma_n(X)\overset{}{\longrightarrow} \pi
	_{n}(X)\overset{h_{n}}{\longrightarrow }%
	H_{n}(X)\rightarrow \cdots
\end{equation}
and proved the following.
\begin{theorem}
	The above sequence is a natural exact sequence, called the certain exact sequence of Whitehead. Furthermore, if $X$ is $(n-1)$-connected, where $n\geq 3$, then
	\begin{equation}\label{x32}
		\Gamma_{n+1}(X)=H_{n}(X)\otimes \Z_{2}.	
	\end{equation}
\end{theorem}

\noindent Notation. We shall denote the sequence (\ref{3}) by
$\mbox{WES}(X)$.
\medskip

Let   $\bf{CW^6_2}$ denote the category of \text{\rm{2}}-connected and \text{\rm{6}}-dimensional CW-complexes such that $H_{3}(X)\otimes\Z_2=0$.  Thus, for an  object $X$ of  $\bf{CW^6_2}$, the sequence (\ref{3}) can be written as 
\begin{equation*}
	\label{333} 
	H_{6}(X)\overset{b_{6}}{\rightarrow }%
	\Gamma_5(X)\overset{}{\rightarrow} \pi
	_{5}(X)\overset{}{\rightarrow }%
	H_{5}(X)\overset{b_{5}}{\rightarrow } \cdots\overset{}{\rightarrow }%
	H_{4}(X)\overset{b_4}{\rightarrow } \Gamma_3(X)\overset{}{\rightarrow} \pi
	_{3}(X)\overset{}{\twoheadrightarrow }%
	H_{3}(X)
\end{equation*}

\noindent Note that since $X$ is \text{\rm{6}}-dimensional we can choose the cellular chain complex of ($C_{*}X,d)$ such that   
$C_{i}X=0$ for all $i\geq 7$ which implies that 
\begin{equation}\label{zx23}
	H_{6}(X)=\ker d_6\,\,\,\,\,\,\,\,\,\,\,\,\,\text{and}\,\,\,\,\,\,\,\,\,\,\,\,\,b_6(z)=\beta(z)\,\,,\,\,\,\forall z\in H_{6}(X).
\end{equation}

The following proposition compute explicitly the groups $\Gamma_{5}(X)$, $\Gamma_{4}(X)$, $\Gamma_{3}(X)$ when  $X$  is an object of  $\bf{CW^6_2}$.
\begin{proposition}\label{tt1}
	If $X$ is an object of  $\bf{CW^6_2}$, then
	\begin{equation*}
		\Gamma_{5}(X)\cong H_{4}(X)\otimes\Z_{2}\oplus \K_{2}(H_{3}(X);\Z)\,\,\,\,\,\,\,\text{and}\,\,\,\,\,\,\,\Gamma_{4}(X)=\Gamma_{3}(X)=0.
	\end{equation*}
	Here $\K_{2}(H_{3}(X);\Z)$ denotes the second homology group  of  $H_{3}(X)$ with integer coefficients.  
\end{proposition} 
\begin{proof}	
	First,	since $X$ is an object of  $\bf{CW^6_2}$, then we can assume that $X^2\sim \star$. Therefore, from the formula (\ref{2}), we deduce that $\Gamma_{3}(X)=0.$ 
	Moreover, by hypothesis we know that $H_{3}(X)\otimes \Z_{2}=0$, hence,  from   (\ref{x32}), we derive that
	$$\Gamma_{4}(X)=H_{3}(X)\otimes \Z_{2}=0$$
	As a result, the sequence (\ref{3}) can be written as follows.
	\begin{equation*}
		\cdots \rightarrow
		H_{4}(X)\overset{b_{4}}{\longrightarrow }%
		\Gamma_4(X)=0\overset{}{\longrightarrow} \pi
		_{4}(X)\overset{h_{4}}{\longrightarrow }%
		H_{4}(X)\overset{b_{3}}{\longrightarrow }	\Gamma_3(X)=0\rightarrow \cdots
	\end{equation*}
	Hence, 
	\begin{equation}\label{33}
		\pi_{4}(X)\cong H_{4}(X)\,\,\,\,\,\,\,\,\,\,,\,\,\,\,\,\,\,\,\,\,\,\,\pi_{3}(X)\cong H_{3}(X).
	\end{equation}
	Next,	in (\cite{BG},  Corollary 2, pp. 170),   it is shown that $\Gamma_{5}(X)$ can be inserted in the following short exact sequence
	$$0\to\pi_{4}(X)\otimes\Z_{2}\oplus \Lambda^2(\pi_{3}(X))\to \Gamma_{5}(X)\to \text{\rm{Tor}}\,(\pi_{3}(X);\Z_2)\to0.$$
	Recall that 
	$$\Lambda^2(\pi_{3}(X))=\pi_{3}(X)\otimes\pi_{3}(X)/(x\otimes x\sim 0),$$ is the exterior square of $\pi_{3}(X)$ (see \cite{BG}, the formula (1.10) for more details.)

	\noindent Furthermore,  since $H_{3}(X)\otimes\Z_2=0$ and $H_{3}(X)\cong \pi_{3}(X)$,  it follows that 	
	$$\text{\rm{Tor}}\,(\pi_{3}(X);\Z_2)=\text{\rm{Tor}}\,(H_{3}(X);\Z_2)=0.$$ Therefore
	$$\Gamma_{5}(X)\cong\pi_{4}(X)\otimes\Z_{2}\oplus \Lambda^2(\pi_{3}(X)).$$
	Next, Theorem 3 in \cite{Mi} asserts that 
	$$ \Lambda^2(\pi_{3}(X))\cong \K_{2}(\pi_{3}(X);\Z).$$
	As a result, using (\ref{33}), we obtain
	$$\Gamma_{5}(X)\cong\pi_{4}(X)\otimes\Z_{2}\oplus \K_{2}(\pi_{3}(X);\Z)\cong H_{4}(X)\otimes\Z_{2}\oplus \K_{2}(\pi_{3}(X);\Z),$$ 
	as wanted.
\end{proof} 
If  $[\alpha]:X\to X$ is a homotopy equivalence  in $\bf{CW^6_2}$, then  $\alpha$ induces an automorphism  $$\Gamma_{5}(\alpha):\Gamma_{5}(X)\to \Gamma_{5}(X),$$
which is the restriction of the homomorphism $\pi_5(\alpha^5):\pi_5(X^5)\to \pi_5(X^5)$ induced by the restriction $\alpha^5:X^5\to X^5$ of the map $\alpha$ to the space $X^5$.
\begin{proposition}\label{pp1}
If  $[\alpha]:X\to X$ is a homotopy equivalence  in $\bf{CW^6_2}$, then $\Gamma_{5}(\alpha)$ satisfies the following relation.
	\begin{equation}\label{z3}
		\Gamma_{5}(\alpha)=H_{4}(\alpha)\otimes id_{\Z_2}\oplus \K_{2}(H_{3}(\alpha);\Z),
	\end{equation}
	where the automorphism $$\K_{2}(H_{3}(\alpha);\Z):\K_{2}(H_{3}(X);\Z)\to \K_{2}(H_{3}(X);\Z),$$
	is induced by the  $H_{3}(\alpha)\in\aut(H_{3}(X))$ on the second homology groups. 
\end{proposition}
\begin{proof}
	Straightforward from the fact that  
	$$	\Gamma_{5}(\alpha):	 H_{4}(X)\otimes\Z_{2}\oplus \K_{2}(H_{3}(X);\Z)\to  H_{4}(X)\otimes\Z_{2}\oplus \K_{2}(H_{3}(X);\Z),$$
	according to Proposition (\ref{tt1}) and the fact that $\alpha$ is a cellular maps applying $X^4$ into  $X^4$ and $X^3$ into  $X^3$.
\end{proof}



\begin{corollary}\label{c01}
	If $X$ is an object of  $\bf{CW^6_2}$, then WES$(X)$ can be written as follows. 
	$$H_{6}(X)\overset{b_{6}}{\to  }\Gamma_{5}(X)\to
	\pi_{5}(X)\overset{h_5}{\twoheadrightarrow} H_{5}(X)\,\,\,\,\,,\,\,\,\,\,\,\,\pi_{4}(X)\cong H_{4}(X)\,\,\,\,\,,\,\,\,\,\,\,\,\pi_{3}(X)\cong H_{3}(X). $$
\end{corollary}
\begin{proof}
	It suffices to apply Proposition \ref{tt1}.
\end{proof}
\begin{corollary}\label{c02}
Let  $[\alpha]:X\to X$ be  a homotopy equivalence  in $\bf{CW^6_2}$ 
	and $\Gamma_{5}(\alpha)$ as in (\ref{z3}), then  $\alpha$ induces the following commutative  diagram. 
	
	\begin{picture}(300,90)(05,10)
		\put(65,80){ $H_{6}(X)\overset{b_{6}}{\longrightarrow  }\Gamma_{5}(X)\longrightarrow 
			\pi_{5}(X)\longrightarrow H_{5}(X)\rightarrow  0$}
		\put(230,76){$\vector(0,-1){50}$}
		\put(70,17){$H_{6}(X)\overset{b_{6}}{\longrightarrow  }\Gamma_{5}(X)\longrightarrow 
			\pi_{5}(X)\longrightarrow H_{5}(X)\rightarrow  0$}
		\put(75,76){$\vector(0,-1){50}$} \put(77,50){\scriptsize
			$H_{6}(\alpha)$} 
		\put(182,50){\scriptsize
			$\pi_{5}(\alpha)$} 
		\put(180,76){$\vector(0,-1){50}$} \put(232,50){\scriptsize
			$H_{5}(\alpha)$}
		\put(132,76){$\vector(0,-1){50}$} \put(134,50){\scriptsize
			$\Gamma_{5}(\alpha)$}
	\end{picture}
\end{corollary}
\begin{proof}
	It suffices to apply Proposition \ref{pp1}.
\end{proof}	
\section{$\Gamma$-homomorphisms }
\subsection{The characteristic extension}
Let $X$ be an object of $\bf{CW^6_2}$ and  let's consider  the homomorphism 
\begin{equation}\label{z1}
	j_5 : \pi_5(X^5)\longrightarrow C_5X= \pi_{5}(X^{5},X^{4}),
\end{equation}
extracted from the following long sequence (see Remark \ref{r5})
$$\cdots\to C_6X\overset{\beta_{6}}{\longrightarrow}\pi_{5}(X^{5})\overset{j_{5}}{\longrightarrow} C_5X\overset{\beta_{5}}{\longrightarrow} \pi_{4}(X^{4})\overset{j_{4}}{\longrightarrow} C_4X\overset{\beta_{4}}{\longrightarrow} \pi_{4}(X^{4})\to\cdots$$
Since by Proposition \ref{tt1}  we have $\Gamma_4(X)=0$, we claim that  $$\ker
\beta_5=\ker d_5.$$
Indeed; we know that $d_5=j_4\circ\beta_5$, therefore $\ker\beta_5\subseteq\ker d_5$.  Next,  if $d_5(x)=0$, then $j_4\circ\beta_5(x)=0$. As a result,  $\beta_5(x)\in\ker j_4=\Gamma_4(X)=0$. Hence, $\beta_5(x)=0$ implying that $\ker d_5\subseteq\ker d\beta_5$. 

\medskip
The sequence (\ref{z1})  gives rise to the short exact
sequence
\begin{equation*}\label{12}
	\Gamma_4(X) \rightarrowtail  \pi_4(X^4) \twoheadrightarrow  
	\ker d_5=\ker \beta_5=\im j_4.
\end{equation*}
As $C_5X$ is a free abelian group, $\ker d_5\subset C_5X$ is also free and the later short exact sequence splits.
So we can choose a section  $\sigma_5 :\ker d_5\longrightarrow\pi_5(X^5)$ of $j_5$ and a splitting
\begin{equation}\label{13}
	\mu_{5} : \pi_5(X^5) \stackrel{\cong}{\rightarrow}\Gamma_5(X) \oplus
	\ker d_5\,\,\,\,\,\,\,\,\,\,\,,\,\,\,\,\,\,\,\,\,\,\,\,\,\mu_5(x)=(x-\sigma_5 \circ j_5(x))\oplus j_5(x).
\end{equation}
\begin{remark}\label{rr2} 
	We have the following clear facts
	\begin{enumerate}
		\item 	From the formula (\ref{13}), it follows that if $x\in\Gamma_5(X)$, then  $\mu_5(x)=x$, i.e.,  the map $\mu_5$ is the inclusion on $\Gamma_5(X) \subset \pi_5(X^5)$. 
		\item Using the same arguments,  we can obtain the following formula 
		$$	\mu_{4} : \pi_4(X^4) \stackrel{\cong}{\rightarrow}\Gamma_4(X) \oplus
		\ker d_4\,\,\,\,,\,\,\,\,\mu_4(x)=(x-\sigma_4 \circ j_4(x))\oplus j_4(x).$$
		and since  $\Gamma_4(X)=0$ (see Proposition \ref{tt1}),  it follows that
		\begin{equation}\label{z13}
			\mu_4: \pi_4(X^4) \stackrel{\cong}{\rightarrow}
			\ker d_4\,\,\,\,\,\,\,\,\,\,\,,\,\,\,\,\,\,\,\,\,\,\,\,\,\mu_4(x)= j_4(x).
		\end{equation}
	where $\sigma_4$ is a chosen section of $j_4$. 
	\end{enumerate}
\end{remark}
\bigskip

Consider the differential of the cellular complex  $d_{6} :
C_{6}X \rightarrow C_5X$. If  $(\im d_6)'\subset C_6(X)$ denotes  a copy isomorphic  to  the free abelian subgroup $\im d_{6} \subset C_5X$, then we get the following direct sum
\begin{equation}\label{18}
	C_{6}X\cong \ker d_6\oplus (\im d_6)'.
\end{equation}
As a result,  using  (\ref{13}) and  (\ref{18}), we  can  rewrite the homomorphism $\beta_{6} : C_{6}X
\rightarrow \pi_5(X^5)$ to get the homomorphism 
\begin{equation*}\label{z9}
	\ker d_{6}\oplus (\im d_{6})'= C_{6}X
	\stackrel{\beta_{6}} {\longrightarrow} \pi_5(X^5) \stackrel{\mu_5}{\longrightarrow}  \Gamma_5(X) \oplus \ker d_5.
\end{equation*}
Then notice that
\begin{equation}\label{x33}
	(\im d_6)'  \stackrel{d_6}{\rightarrowtail}\mbox{ker
	}d_5 \rightarrow H_5(X),
\end{equation}
can be chosen as  a free resolution  of $H_5(X)$. Thus, we define
$$\mbox{Ext}^1_{\mathbb{Z}}(G_5,  \coker \delta)=\frac{\Hom(	(\im d_{6} )',\coker \delta)}{\im (d_6)^*},$$
where $(d_6)^*:\Hom(	\ker d_5,\coker \delta)\to \Hom(	(\im d_{6} )',\coker \delta)$.
\medskip

Due to (\ref{13}), it follows  that for every $l\in (\im d_{6})'$ we have 
$$\mu_5\circ \beta_{6}(l)=\beta_{6}(l) - \sigma_5 \circ d_{6}(l),$$ 
implying that 
$$j_5\circ\mu_5\circ \beta_{6}(l)=j_5\circ\beta_{6}(l) - j_5\circ\sigma_5 \circ d_{6}(l)=d_{6}(l)-d_{6}(l)=0,$$ 
therefore $\mu_5\circ \beta_{6}((\im d_{6})')\subset \Gamma_5(X)$. 
\medskip 

Thus,  we define  $\theta_X:(\im d_{6})'\overset{}{\rightarrow} \Gamma_6(X)\overset{}{\rightarrow} \mathrm{coker}\,b_{6}$ by 
\begin{equation}\label{x15}
	\theta_X=pr\circ(\beta_{6} - \sigma_5 \circ d_{6}).
\end{equation}
where $pr:\Gamma_5(X)\twoheadrightarrow \mathrm{coker}b_{6}$ is the projection. Hence, taking into account  the resolution (\ref{x33}), we obtain  
the extension class
\begin{equation}\label{x51}
	[\theta_X]\cong \Big[\frac{\coker b_{6}\oplus \ker d_{6}}{\im \theta_X\oplus\im d_{6}}\Big] \in \mbox{Ext}^1_{\mathbb{Z}}(H_5(X),  \mathrm{coker}\,b_{6}).
\end{equation}
\begin{definition}
	\label{d6} The class $[\phi_X] $ is called  the
	characteristic  extension of the CW-complex $X$.
\end{definition}
It is worth noting that if $\mathrm{Ext}(H_5(X), \coker b_{6})$   denotes  the abelian group of the extensions classes of  $\coker\,b_6$ by $H_5(X)$, then it is well-know that
$$\mathrm{Ext}(H_5(X), \coker b_{6})\cong\mbox{Ext}^1_{\mathbb{Z}}(H_5(X), \coker b_{6}),$$  
and using (\ref{x51}), we derive the following correspondence   
\begin{equation}\label{zz1}
	\mbox{Ext}^1_{\mathbb{Z}}(H_5(X), \coker b_{6}) \ni[\theta_X]\longleftrightarrow [\pi_{5}(X)]\in \mathrm{Ext}(H_5(X), \coker b_{6}),
\end{equation}
where  $[\pi_{5}(X)]$ is the extension class  representing the short exact sequence
$$\coker b_{6}\rightarrowtail
\pi_{5}(X)\twoheadrightarrow H_5(X),$$  extracted for the
Whitehead exact sequence given in Corollary \ref{c01}. 	

\begin{remark}
	\label{r2}
	Let $X$ be  an  objects of  $\bf{CW^6_2}$ and  $\omega :X^5 \to  X^5$ a map. If  $\xi_*=C_*(\omega)\colon   C_*X^5\rightarrow  C_*X^5$ is  the chain map induced by $\omega$,  then  the following two diagrams commute
	\begin{equation}\label{z19}
		\begin{picture}(300,80)(10,40)
			\put(25,100){ $ \pi_5(X^5)
				\hspace{1mm}\vector(1,0){40}\,C_{5}X^5$}
			\put(15,76){\scriptsize $\pi_{5}(\omega)$} \put(112,76){\scriptsize $\xi_5$}
			\put(37,97){$\vector(0,-1){39}$} \put(110,96){$\vector(0,-1){39}$}
			\put(80,103){\scriptsize $j_5$} \put(80,52){\scriptsize
				$\scriptsize j_5$} \put(23,48){ $ \pi_{5}(X^5)\hspace{1mm}\vector(1,0){40}\,C_{5}X^5$} 
			
			\put(205,100){ $ \pi_4(X^4)
				\hspace{1mm}\vector(1,0){40}\,C_{4}X^5$}
			\put(195,76){\scriptsize $\pi_{4}(\omega)$} \put(292,76){\scriptsize $\xi_4$}
			\put(217,97){$\vector(0,-1){39}$} \put(290,96){$\vector(0,-1){39}$}
			\put(260,103){\scriptsize $j_4$} \put(260,52){\scriptsize
				$\scriptsize j_4$} \put(203,48){ $ \pi_{4}(X^4)\hspace{1mm}\vector(1,0){40}\,C_{4}X^5$} 
		\end{picture}
	\end{equation}
Moreover, using the splitting $\mu_5$, given in (\ref{13}),  we obtain the following diagram	
	\begin{equation*}
		\begin{picture}(300,80)(30,40)
			\put(75,100){ $ \pi_5(X^5)
				\hspace{1mm}\vector(1,0){110}\,\Gamma_{5}(X) \oplus \ker\beta_5$}
			\put(89,76){\scriptsize $\pi_{5}(\omega)$} \put(250,76){\scriptsize $\Gamma_{5}(\omega) \oplus  \xi_5$}
			\put(87,97){$\vector(0,-1){39}$} \put(248,96){$\vector(0,-1){39}$}
			\put(170,103){\scriptsize $\mu_5$} \put(165,52){\scriptsize
				$\scriptsize \mu_5$} \put(73,48){ $ \pi_{5}(X^5)\hspace{1mm}\vector(1,0){110}\,\Gamma_{5}(X) \oplus \ker\beta'_5$} 
		\end{picture}
	\end{equation*}
	Taking into account the commutativity of the diagram (\ref{z19}), 	 it is easy to check that
	\begin{equation} \label{x5}
		(\Gamma_{5}(\omega) \oplus  \xi_5)\circ \mu_5 - \mu_5\circ \pi_5(\omega) =  \pi_5(\omega) \circ
		\sigma_5 \circ j_5 - \sigma_5 \circ j_5 \circ \pi_5(\omega).
	\end{equation}
	Next, if we consider the following diagram 
	
	\begin{picture}(300,80)(10,40)
		\put(75,100){ $ \pi_4(X^4)
			\hspace{1mm}\vector(1,0){110}\,\ker\beta_4$}
		\put(89,76){\scriptsize $\pi_{4}(\omega)$} \put(230,76){\scriptsize $  \xi_4$}
		\put(87,97){$\vector(0,-1){39}$} \put(228,96){$\vector(0,-1){39}$}
		\put(165,103){\scriptsize $\mu_4$} \put(165,52){\scriptsize
			$\scriptsize \mu_4$} \put(73,48){ $ \pi_{4}(X^4)\hspace{1mm}\vector(1,0){110}\, \ker\beta_4$} 
	\end{picture}
	
	\noindent Using   (\ref{z13}),  the diagram (\ref{z19})  and   $\Gamma_{4}(X)=0$ (see Proposition (\ref{pp1})),   we obtain
	\begin{equation*}
		\xi_4\circ \mu_4(x)-\mu_4\circ \pi_{4}(\omega)(x)=\xi_4\circ j_4(x)-j_4\circ\pi_{4}(\omega)(x)=0,
	\end{equation*}
	for all $x\in \pi_4(X^4).$
\end{remark}
\subsection{$\Gamma$-automorphisms }
\begin{definition}
	\label{d1}	Given an object      $X$  in the category  $\bf{CW^6_2}$ and    $f_{*}\in\aut (H_{*}(X))$.  We say that  $f_{*}$ is a $\Gamma$-automorphism of  WES$(X)$, if 
  there exists  $\phi\in \aut(\pi_{5}(X))$  making the following diagram commutes:
	\begin{equation}\label{z11}
		\begin{picture}(300,100)(10,-05)
			\put(65,80){ $H_{6}(X)\overset{b_{6}}{\longrightarrow  }\Gamma_{5}(X)\longrightarrow 
			\pi_{5}(X)\twoheadrightarrow H_{5}(X)$}
		\put(175,76){$\vector(0,-1){50}$}
		\put(70,17){$H_{6}(X)\overset{b_{6}}{\longrightarrow  }\Gamma_{5}(X)\longrightarrow 
			\pi_{5}(X)\twoheadrightarrow H_{5}(X)$}
		\put(75,76){$\vector(0,-1){50}$} \put(77,50){\scriptsize
			$f_{6}$} 
		\put(177,50){\scriptsize
			$\phi$} 
		\put(225,76){$\vector(0,-1){50}$} \put(227,50){\scriptsize
			$f_{5}$}
		\put(130,76){$\vector(0,-1){50}$} \put(132,50){\scriptsize
			$\gamma$}
	\end{picture}
	\end{equation}
where 
\begin{itemize}
	\item $\gamma= f_{4}\circ \otimes id_{\Z_2}\oplus \K_{2}\big( f_{3};\Z\big)$.
	\item $\Gamma_{5}(X)=H_{4}(X)\otimes id_{\Z_2}\oplus \K_{2}(H_{3}(X);\Z),$
	\item  The homomorphism $\K_{2}(f_{3}):\K_{2}(H_{3}(X);\Z)\to \K_{2}(H_{3}(X);\Z)$ 
	is induced by the homomorphism $f_{3}$ on the second homology groups. 
\end{itemize}
\end{definition}
\begin{example}
	\label{e1}
Let  $X$  be   an    object of $\bf{CW^6_2}$.   If  $\alpha:X\to X$ is  a homotopy equivalence, then
	$H_{*}(\alpha)$ is a $\Gamma$-automorphism according to corollary \ref{c02}. 
\end{example}
\begin{remark}\label{r02}
The  diagram (\ref{z11}) implies that   the following  diagram is commutative.
\begin{equation}\label{z0}
	\begin{picture}(300,80)(-35,10)
		\put(65,80){ $\coker b_6\rightarrowtail \pi_{5}(X)\twoheadrightarrow H_{5}(X)$}
		\put(65,16){$\coker b_6\rightarrowtail \pi_{5}(X)\twoheadrightarrow H_{5}(X)$}
		\put(75,76){$\vector(0,-1){50}$} 
		\put(127,50){\scriptsize
			$\phi$} 
		\put(165,76){$\vector(0,-1){50}$} \put(167,50){\scriptsize
			$f_{5}$}
		\put(125,76){$\vector(0,-1){50}$}
		\put(77,50){\scriptsize
			$\tilde{\gamma}$}
	\end{picture}
\end{equation}

\noindent where $\tilde{\gamma}$ is  induced by  $\gamma$ and the  diagram (\ref{z11}).  Note that  the following diagram commutes.
\begin{equation}\label{x50}
	\begin{picture}(300,80)(-10,40)
		\put(75,100){ $ \Gamma_5(X)
			\hspace{1mm}\vector(1,0){110}\,\Gamma_5(X)$}
		\put(89,76){\scriptsize $pr$} \put(230,76){\scriptsize $pr$}
		\put(87,97){$\vector(0,-1){39}$} \put(228,96){$\vector(0,-1){39}$}
		\put(160,103){\scriptsize $\gamma$} \put(160,52){\scriptsize
			$\scriptsize \tilde{\gamma}$} \put(75,48){ $\coker b_{6}\hspace{1mm}\vector(1,0){100}\hspace{1mm}\coker b_{6}$} 
	\end{picture}
\end{equation}
Let us denote by
$$ (f_5)^* : \mbox{Ext}(H_{5}(X),  \coker b_6)\rightarrow \mbox{Ext}(H_{5}(X),  \coker b_6),$$
the homomorphism  induced by $f_5$ and by  
$$ (\tilde{\gamma})_* : \mbox{Ext}(H_{5}(X), \coker b_6) \rightarrow (H_{5}(X), \coker b_6),$$ 
the homomorphism induced $\tilde{\gamma} $. So, the commutativity of the diagram (\ref{z0}) implies   the following  formula.
\begin{equation}\label{y15}
	(f_5)^*([\pi_{5}(X)] )= (\tilde{\gamma})_*([\pi_{5}(X)]).
\end{equation}
It is worth mentioning that the formula (\ref{y15})  means the following. First,  by virtue of the homotopy extension theorem \cite{ma}, there exists  a chain map   $\xi_{*}:(C_{*},d)\to (C_{*},d')$ such that $H_{i}(\xi_{*})=f_i$ for $ i=3,4,5,6$.  Next, from (\ref{x33}) we know that 
$$(\im d_{6} )' \stackrel{d_6}{\rightarrowtail}\ker d_5 \twoheadrightarrow H_{5}(X),$$
is a free resolution of $H_{5}(X)$. Moreover, according to (\ref{zz1}),  to the given extension $[\pi_{5}(X)]$  corresponds the characterization  class
$$[\theta_X]\in \mbox{Ext}^1_{\mathbb{Z}}(H_{5}(X), \coker b_6),$$
and to the homomorphisms $(f_5)^*$,  $(\tilde{\gamma})_*$  correspond the following two 
diagrams.

\begin{picture}(300,140)(60,-10)
	\put(100,98){$\vector(0,-1){42}$} 
	\put(90,102){$(\im d_{6})' \stackrel{d_6}{\rightarrowtail}\ker d_5 \twoheadrightarrow H_{5}(X)$} \put(90,45){$\coker b_6$}
	\put(102,78){\scriptsize $\theta_X$}
	\put(90,-8){$ \coker b_6$}	\put(100,43){$\vector(0,-1){42}$}
	\put(100,98){$\vector(0,-1){42}$} 
	\put(365,98){$\vector(0,-1){42}$} 
	\put(367,78){\scriptsize $f_5$} 
	\put(270,102){$(\im d_{6})'\stackrel{d_6}{\rightarrowtail}\ker d_5 \twoheadrightarrow H_{5}(X)$} 
	\put(102,22){\scriptsize $\tilde{\gamma} $}	\put(282,78){\scriptsize $\xi_{6}$}
	\put(270,-8){$\coker b_6$}	\put(280,43){$\vector(0,-1){42}$}
	\put(270,45){$(\im d_{6})'\stackrel{d_6}{\rightarrowtail}\ker d_5 \twoheadrightarrow H_{5}(X)$} 
	\put(280,98){$\vector(0,-1){42}$}
	\put(282,22){\scriptsize $\theta_X$}
\end{picture}

\noindent  where
\begin{equation*}
	(f_5)^*([\pi_{5}(X)])=[\theta_X\circ \xi_{6}]\,\,\,\,\,\,\,\,\,\,\,\,\,\,\,\,\,,\,\,\,\,\,\,\,\,\,\,\,\,\,\,(\tilde{\gamma} )_*([\pi_{5}(X)])=[\tilde{\gamma} \circ \theta_X],
\end{equation*}  
implying that
\begin{equation*}
	(f_5))^*([\pi_{5}(X)])-(\tilde{\gamma} )_*([\pi_{5}(X)])=[\theta_X\circ \xi_{6}-\tilde{\gamma} \circ \theta_X].
\end{equation*}
Hence, the  relation  (\ref{y15}) is equivalent to the existence of a homomorphism $g$
$$\ker d_5\stackrel{g}{\longrightarrow} \Gamma_5(X)\stackrel{pr}{\longrightarrow}\coker b_{6}\to 0 $$ 
satisfying the relation
\begin{equation}\label{x3}
	\theta_X\circ \xi_{6}-\tilde{\gamma} \circ \theta_X=pr\circ g\circ d_6.
\end{equation}
\end{remark} 
\begin{lemma}\label{l1}
Let  $\omega:X^5\to X^5$ be  a map. If $g\in \H(\ker d_5;\Gamma_{5}(X))$,  then there exists a map $\eta:X^5\to X^5$ such that
$$\pi_{5}(\eta)=\pi_{5}(\omega)+g\circ j_5. $$		
where $j_5$ is given in \rm{(\ref{z1})}. Moreover, $C_{*}(\eta)=C_{*}(\omega):C_{*}(X^5)\to C_{*}(X^5)$.
\end{lemma}
	\begin{proof}
Since $\ker d_5$ is a free abelian group, there exists a homomorphism $\tilde{g}$ making the following diagram commutes

\begin{equation*} \label{x77}
	\begin{picture}(180,80)(40,35)
		\put(45,98){$\vector(0,-1){40}$}
		\put(37,102){$
			\ker d_5$}
		\put(37,45){$ \pi_5(X^{4})\vector(1,0){118}\vector(1,0){2}\,\im \pi_{5}(i)=\Gamma_{5}(X) \subset\pi_5(X^{5})$} \put(47,75){\scriptsize
			$\tilde{g}$}
		\put(52,98){$\vector(3,-1){138}$} \put(127,75){\scriptsize $g$}
		\put(117,48){\scriptsize $\pi_{5}(i)$}	
	\end{picture}
\end{equation*}
 where  $\pi_{5}(i):\pi_{5}(X^4)\rightarrow \pi_{5}(X^5)$ is induced by the inclusion $i:X^4\hookrightarrow X^5$. 
 Here we invoke the formula (\ref{2}).      Using the decomposition     $C_{5}X\cong \ker d_5\oplus (\im d_5)'$, given $(\ref{18})$,  let  $\B_{1}$ be a basis for $\ker d_5$ and $\B_{2}$  a basis for $(\im d_5)'$. 

Let us represent $X^5$ as a union of its 4-skeleton  $X^4$ and the 5-cells. i,e,. 
$$X^5=X^4\cup(\underset{ z\in \B_{1}}{\cup}e_z^{5})\cup(\underset{ l\in \B_{2}}{\cup}e_l^{5}).$$
Recall that $\pi_5(X^5, X^{4})$ is the free $\mathbb{Z}$-module generated by the $5$-cells of $X$. Therefore, if $e_z^5$ is a 5-cell, then the differential   $d_5:C_5(X)=\pi_{5}(X^5, X^4)\to C_4(X)=\pi_{4}(X^4, X^3)$ is defined by $d_5(e_z^5)=[\tau_z/_{S^4}]$, where $\tau_z$ is the attaching map of the 5-cell $e_z^5$.

Thus, the homomorphism $\tilde{g}:\ker d_5\stackrel{}{\longrightarrow}  \pi_5(X^4)$ allows us to define a  map
$$\nu:\underset{ z\in \B_{1}}{\cup}e_z^{5}\to X^4,$$
such that $[\nu(e_z^{5})]=\tilde{g}(e_z^5)$ which implies that
$$g(e_z^{5})=\pi_{5}(i)\circ \tilde{g}(e_z^{5}).$$ 
Next,  let us define  a new map $\eta:X^5\to X^5$ by its following restrictions 
$$\eta/_{X^4}=\omega\,\,\,\,\,\,\,\,\,\,\,\,\,,\,\,\,\,\,\,\,\,\,\,\,\eta/_{\underset{ l\in \B_{2}}{\cup}e_l^{5}}=\omega\,\,\,\,\,\,\,\,\,\,\,\,\,,\,\,\,\,\,\,\,\,\,\,\,\eta/_{\underset{ z\in \B_{1}}{\cup}e_z^{5}}=\omega\vee i\circ\nu.$$


\noindent Clearly, the homomorphism   $\pi_{5}(\eta):\pi_{5}(X^5)\to \pi_{5}(X^5)$ satisfies  
\begin{equation*}
\pi_{5}(\eta)=\pi_{5}(\omega)+[i\circ\nu(e_z^{5})]=\pi_{5}(\omega)+g\circ j_5.
\end{equation*}
Recall that  we have $\pi_{5}(X^5)\stackrel{j_5}{\longrightarrow} \ker d_5\stackrel{g}{\longrightarrow} \Gamma_{5}(X) \subset\pi_5(X^{5})$.
\end{proof}
\begin{remark}
	\label{r01}   
Later on we need the  following elementary fact.

\noindent Let $X$  be  an object of $\bf{CW^6_2}$ and  $\tau :X^5 \to  X^5$ a map between its 5-skeleton . If $\rho$ is a  
homomorphism making the following diagram commutes

\begin{picture}(300,80)(10,40)
	\put(75,100){ $ C_{6}X
		\hspace{1mm}\vector(1,0){140}C_{6}X$}
	\put(89,76){\scriptsize $\beta_{6}$} \put(250,76){\scriptsize $\beta_{6}$}
	\put(87,97){$\vector(0,-1){39}$} \put(248,96){$\vector(0,-1){39}$}
	\put(165,103){\scriptsize $\rho$} \put(160,52){\scriptsize
		$\scriptsize \pi_{5}(\tau)$} \put(73,48){ $ \pi_{5}(X^5)\hspace{1mm}\vector(1,0){130}\hspace{1mm}\pi_{5}(X^5)$} 
\end{picture}

\noindent then $\tau$ can be extended to a homotopy equivalence  $\tilde{\tau}:X \to X$ with $C_{6}(\tilde{ \tau }) = \rho$.
\end{remark}
\subsection{The group $\Gamma\mathcal{G}(X)$}
Let $X$  be  an object of $\bf{CW^6_2}$ and  let   $\Gamma\mathcal{G}(X)$ denote the set of all the  $\Gamma$-automorphisms of  WES$(X)$.  
\begin{proposition}
	\label{p2}                                                    
	The set $\Gamma\mathcal{G}(X)$ is a group.
\end{proposition}
\begin{proof}
	First let us prove that if  $f_{*},f'_{*}\in\Gamma\mathcal{G}(X)$, then the composition
	$$f'_{*}\circ f_{*}\in \Gamma\mathcal{G}(X).$$
	Indeed,  from definition \ref{d1} we derive that there exist  $\phi,\phi'\in\aut(\pi_{5}(X))$ making the following  diagrams commutes
	
	\begin{picture}(300,160)(-05,-55)
		\put(65,80){ $H_{6}(X)\overset{b_{6}}{\longrightarrow  }\Gamma_{5}(X)\longrightarrow 
			\pi_{5}(X)\twoheadrightarrow H_{5}(X)$}
		\put(175,76){$\vector(0,-1){50}$}
		\put(70,17){$H_{6}(X)\overset{b_{6}}{\longrightarrow  }\Gamma_{5}(X)\longrightarrow 
			\pi_{5}(X)\twoheadrightarrow H_{5}(X)$}
		\put(75,76){$\vector(0,-1){50}$} \put(77,50){\scriptsize
			$f_{6}$} 
		\put(177,50){\scriptsize
			$\phi$} 
		\put(225,76){$\vector(0,-1){50}$} \put(227,50){\scriptsize
			$f_{5}$}
		\put(130,76){$\vector(0,-1){50}$} \put(132,50){\scriptsize
			$\gamma$}
		\put(175,14){$\vector(0,-1){50}$}
		\put(70,-45){$H_{6}(X)\overset{b_{6}}{\longrightarrow  }\Gamma_{5}(X)\longrightarrow 
			\pi_{5}(X)\twoheadrightarrow H_{5}(X)$}
		\put(75,14){$\vector(0,-1){50}$} \put(77,-15){\scriptsize
			$f'_{6}$} 
		\put(177,-15){\scriptsize
			$\phi'$} 
		\put(225,14){$\vector(0,-1){50}$} \put(227,-15){\scriptsize
			$f'_{5}$}
		\put(130,14){$\vector(0,-1){50}$} \put(132,-15){\scriptsize
			$\gamma'$}
	\end{picture}	
	
	\noindent where  $$\gamma=f_{4}\otimes id_{\Z_2}\oplus \K_{2}( f_{3};\Z)\,\,\,\,\,\,\,\,\,\,,\,\,\,\,\,\,\,\,\,\,\gamma'= f'_{4}\otimes id_{\Z_2}\oplus \K_{2}(f'_{3};\Z).$$
	Here we invoke the diagram (\ref{z11}).  As a result, the following diagram is obviously commutative
	
	\begin{picture}(300,90)(-25,10)
		\put(65,80){ $H_{6}(X)\overset{b_{6}}{\longrightarrow  }\Gamma_{5}(X)\longrightarrow 
			\pi_{5}(X)\twoheadrightarrow H_{5}(X)$}
		\put(175,76){$\vector(0,-1){50}$}
		\put(70,17){$H_{6}(X)\overset{b_{6}}{\longrightarrow  }\Gamma_{5}(X)\longrightarrow 
			\pi_{5}(X)\twoheadrightarrow H_{5}(X)$}
		\put(75,76){$\vector(0,-1){50}$} \put(77,50){\scriptsize
			$f'_{6}\circ f_6$} 
		\put(177,50){\scriptsize
			$\phi'_{5}\circ \phi_{5}$} 
		\put(225,76){$\vector(0,-1){50}$} \put(227,50){\scriptsize
			$f'_{5}\circ f_5$}
		\put(130,76){$\vector(0,-1){50}$} \put(132,50){\scriptsize
			$\gamma'\circ\gamma$}
	\end{picture}
	
	\noindent where $\gamma'\circ \gamma= f'_{4}\circ f_{4}\otimes id_{\Z_2}\oplus \K_{2}( f'_{3}\circ f_{3};\Z)$. Hence, 
	  $f'_{*}\circ f_{*}\in \Gamma\mathcal{G}(X)$.
	\medskip
	
	Finally, if  $ f_{*}\in \Gamma\mathcal{G}(X)$, then there is  $\phi\in\aut(\pi_{4}(X))$ making the  diagram $(\ref{z11})$.  By definition $f_{*}$ is a graded automorphism,  so  is   $f^{-1}_{*}$   and     the following diagram is obviously commutative
	
	\begin{picture}(300,90)(-05,10)
		\put(65,80){ $H_{6}(X)\overset{b_{6}}{\longrightarrow  }\Gamma_{5}(X)\longrightarrow 
			\pi_{5}(X)\twoheadrightarrow H_{5}(X)$}
		\put(175,76){$\vector(0,-1){50}$}
		\put(70,17){$H_{6}(X)\overset{b_{6}}{\longrightarrow  }\Gamma_{5}(X)\longrightarrow 
			\pi_{5}(X)\twoheadrightarrow H_{5}(X)$}
		\put(75,76){$\vector(0,-1){50}$} \put(77,50){\scriptsize
			$f^{-1}_{6}$} 
		\put(177,50){\scriptsize
			$\phi^{-1}$} 
		\put(225,76){$\vector(0,-1){50}$} \put(227,50){\scriptsize
			$f^{-1}_{5}$}
		\put(130,76){$\vector(0,-1){50}$} \put(132,50){\scriptsize
			$\gamma^{-1}$}
	\end{picture}

	\noindent  where  $\gamma^{-1}= f^{-1}_{4}\circ\otimes id_{\Z_2}\oplus \K_{2}( f^{-1}_{3};\Z)$. As a result, $ f^{-1}_{*}\in \Gamma\mathcal{G}(X).$
\end{proof}
\section{Main results}
Recall that $\E(X)$ denotes  the  group of  self-homotopy equivalences of  $X$ and  $\E_{*}(X)$ denotes its subgroup of the elements that induce  the identity on $H_{*}(X)$.  It worth mentioning  that Example \ref{e1} allows to define the following map
 $$\Psi:\E(X)\overset{}{\longrightarrow  } \Gamma\mathcal{S}(X) \,\,\,\,\,\,\,\,\,\,\,\,\,,\,\,\,\,\,\,\,\,\,\,\,\Psi([\alpha]) =H_{*}(\alpha),$$
and an easy  computation shows
$$\Psi([\alpha].[\alpha'])=\Psi([\alpha\circ\alpha'])=H_{*}(\alpha\circ\alpha')=H_{*}(\alpha)\circ H_{*}(\alpha')=\Psi([\alpha])\circ \Psi([\alpha']),$$
implying  that $\Psi$ is a homomorphism of groups. 
\medskip

The first main result  of this paper can be stated as follows.
\begin{theorem}
	\label{t4}                                                    
The homomorphism $\Psi$  is surjective.    
\end{theorem}
\begin{proof}
Let  $f_{*}\in \aut(H_{*}(X))$ be  a  $\Gamma$-automorphism of  WES$(X)$. Recall that 
$$WES(X)\,\,:\,\,H_{6}(X)\overset{b_{6}}{\longrightarrow  }\Gamma_{5}(X)\longrightarrow 
\pi_{5}(X)\twoheadrightarrow H_{5}(X).$$
By definition we have 
$$\pi_{4}(X)=
H_{4}(X)\overset{f_4}{\longrightarrow} H_{4}(X)=
\pi_{4}(X) \,\,\,\,\,\,\,\,\,,\,\,\,\,\,\,\,\,\,\pi_{3}(X)=
H_{4}(X)\overset{f_3}{\longrightarrow}  H_{3}(X)=
\pi_{3}(X),$$
and there exists $\phi\in\aut(\pi_5(X))$ making the following  diagram  commutes.
\begin{equation}\label{y11}
	\begin{picture}(300,90)(-05,10)
	\put(65,80){ $H_{6}(X)\overset{b_{6}}{\longrightarrow  }\Gamma_{5}(X)\longrightarrow 
		\pi_{5}(X)\twoheadrightarrow H_{5}(X)$}
	\put(175,76){$\vector(0,-1){50}$}
	\put(70,17){$H_{6}(X)\overset{b_{6}}{\longrightarrow  }\Gamma_{5}(X)\longrightarrow 
		\pi_{5}(X)\twoheadrightarrow H_{5}(X)$}
	\put(75,76){$\vector(0,-1){50}$} \put(77,50){\scriptsize
		$f_{6}$} 
	\put(177,50){\scriptsize
		$\phi$} 
	\put(225,76){$\vector(0,-1){50}$} \put(227,50){\scriptsize
		$f_{5}$}
	\put(130,76){$\vector(0,-1){50}$} \put(132,50){\scriptsize
		$\gamma$}
\end{picture}
\end{equation}

\noindent where
\begin{equation}\label{x36}
	\gamma=f_4\otimes id\oplus \K_2(f_3;\Z).
\end{equation}
Next,  by virtue of the homotopy extension theorem \cite{ma}, there is  a  chain map $\xi_{*}:(C_{*}(X),d)\to (C_{*}(X),d)$ such that $H_{*}(\xi_{*})=f_{*}$. 
Recall that $H_{*}(C_{*}(X))=H_{*}(X).$
	
Since $\xi_{*}$ is a chain map and $X^2=\star$,  the following diagram commutes. 

\begin{picture}(300,80)(-05,40)
	\put(75,100){ $ C_{4}X
		\hspace{1mm}\vector(1,0){140}C_{4}X$}
	\put(89,76){\scriptsize $d_{4}$} \put(250,76){\scriptsize $d_{4}$}
	\put(87,97){$\vector(0,-1){39}$} \put(248,96){$\vector(0,-1){39}$}
	\put(165,103){\scriptsize $\xi_4$} \put(160,52){\scriptsize
		$\scriptsize \pi_{3}(\omega )=\xi_3$} \put(57,48){ $C_{3}X= \pi_{3}(X^3)\hspace{1mm}\vector(1,0){90}\hspace{1mm}C_{3}X=\pi_{3}(X^3)$} 
\end{picture}

\noindent Applying Remark \ref{r01}, we get  a map  $\omega:X^4\to X^4$ 
such that the chain map $C_{*}(\omega) :C_{*}(X^4)\to C_{*}(X^4)$ satisfies the relation $C_{\leq 4}(\omega)=\xi_{\leq 4}$. 

\noindent 	Now, let us consider the following diagram

\begin{picture}(300,80)(-10,35)
	\put(75,100){ $ C_{5}X
		\hspace{1mm}\vector(1,0){140}C_{5}X$}
	\put(89,76){\scriptsize $\beta_{5}$} \put(250,76){\scriptsize $\beta_{5}$}
	\put(87,97){$\vector(0,-1){39}$} \put(248,96){$\vector(0,-1){39}$}
	\put(165,103){\scriptsize $\xi_5$} \put(160,52){\scriptsize
		$\scriptsize \pi_{4}(\omega )$} 
	\put(72,48){ $ \pi_{4}(X^4)\hspace{1mm}\vector(1,0){130}\hspace{1mm}\pi_{4}(X^4)\overset{\mu_4}{\longrightarrow} \ker \beta_4$}
\end{picture}

\noindent Using the formula (\ref{z13}), the diagram (\ref{z19}),  an easy computation shows that 
$$\mu_4\circ  \pi_{4}(\omega )\circ\beta_5=j_4\circ\pi_{4}(\omega )\circ\beta_5=\xi_4\circ j_4\circ\beta_5=\xi_4\circ d_5,$$
$$\mu_4\circ  \beta_5\circ \xi_5=j_4\circ\beta_5\circ \xi_5=d_5\circ \xi_5.$$
Hence, $\mu_4\circ  \pi_{4}(\omega )\circ\beta_5-\mu\circ  \beta_5\circ \xi_5=0.$ 
As $\mu_4$ is an isomorphism, it follows  that 
$$ \pi_{4}(\omega )\circ\beta_5-\beta_5\circ \xi_4=0.$$
Applying again Remark \ref{r01}, we can extend the map $\omega$ to obtain a map (which we denote also by $\omega$) 
$\omega:X^5\to X^5$ 
such that the chain map $C_{*}(\omega) :C_{*}(X^5)\to C_{*}(X^5)$  satisfies  $C_{\leq 5}(\omega)=\xi_{\leq 5}$ . i.e,.
\begin{equation}\label{x30}
	H_{4}(\omega)=f_4\,\,\,\,\,\,\,\,\,\,\,\,\,\,\,\,,\,\,\,\,\,\,\,\,\,\,\,\,\,\,\,\,\,\,H_{3}(\omega)=f_3.
\end{equation}
Now, by combining Proposition \ref{pp1}, the relations (\ref{x36})  and (\ref{x30}),  we derive the following identification.
\begin{equation}\label{x4}
	\Gamma_{5}(\omega)=f_4\otimes id\oplus H_2(f_3;\Z)=\gamma.
\end{equation}
Next, let us  consider the following diagram

\begin{picture}(300,180)(-20,-65)
	\put(75,100){\scriptsize $ C_{6}X
		\hspace{1mm}\vector(1,0){145}\,C_{6}X$}
	\put(89,76){\scriptsize $\beta_{6}$} \put(250,76){\scriptsize $\beta_{6}$}
	\put(87,97){$\vector(0,-1){39}$} \put(248,96){$\vector(0,-1){39}$}
	\put(165,103){\scriptsize $\xi_6$} \put(160,52){\scriptsize
		$\scriptsize \pi_{5}(\omega )$} \put(75,48){\scriptsize $ \pi_{5}(X^5)\hspace{1mm}\vector(1,0){132}\hspace{1mm}\pi_{5}(X^5)$} 
	\put(65,-4){\scriptsize $\Gamma_{5}(X)\oplus \ker d_5\hspace{1mm}\vector(1,0){105}\hspace{1mm}
		\Gamma_{5}(X)\oplus \ker d_5$} \put(89,24){\scriptsize
		$\mu_{5}$} \put(250,24){\scriptsize $\mu_{5}$}
	\put(79,24){\scriptsize $\cong$} \put(241,24){\scriptsize $\cong$}
	\put(87,45){$\vector(0,-1){41}$} \put(248,45){$\vector(0,-1){41}$}
	\put(150,-1){\scriptsize $\Gamma_{5}(\omega)\oplus
		\xi_{5}$}
	\put(-20,-50){\scriptsize $ C_{5}X\supseteq\ker d_5
		\hspace{1mm}\vector(1,0){298}\hspace{1mm} \ker d_5\subseteq C_{5}X$}
	\put(78,46){$\vector(-2,-3){59}$}
	\put(82,95){$\vector(-2,-3){92}$} \put(32,35){\scriptsize
		$d_{6}$}\put(42,5){\scriptsize $j_{5}$}
	\put(260,42){$\vector(1,-1){83}$}
	\put(255,95){$\vector(3,-4){103}$}\put(170,-47){\scriptsize
		$\xi_{5}$} \put(302,35){\scriptsize
		$d_{6}$}\put(302,5){\scriptsize $j_{5}$}
	\put(335,-42){$\vector(-2,2){85}$}
	\put(299,-20){\scriptsize $\sigma_{5}$}
	\put(235,-37){\scriptsize $\coker b_{5}$}
	\put(248,-6){$\vector(0,-1){25}$}
	\put(250,-20){\scriptsize $pr$}
\end{picture}

\noindent In this diagram,    the two
side triangles and  the small and large trapezoids are commutative.  The rectangles are
not commutative. But for every $z\in \ker d_6\subset C_6X$ we have
\begin{equation}\label{zx22}
	\pi_5(\omega)\circ\beta_{6}(z)-\beta_{6}\circ \xi_6(z)=\Gamma_5(\omega)\circ b_6(z)-b_{6}\circ f_6(z)=0
\end{equation}
Here we invoke the diagram (\ref{y11}), the relations (\ref{zx23}) and   (\ref{x4}).
\smallskip

On  one hand, by (\ref{x5})  we have
\begin{equation}\label{x6}
	pr\circ\big((\Gamma_{5}(\omega) \oplus  \xi_5)\circ \mu_{5}\circ\beta_{6} - \mu_{5}\circ \pi_5(\omega)\circ\beta_{6}\big)= pr\circ\big(\pi_5(\omega) \circ
	\sigma_{5} \circ j_5\circ\beta_{6} - \sigma_{5} \circ j_5 \circ \pi_5(\omega)\circ\beta_{6}\big)
\end{equation}
On the other hand, using simultaneously (\ref{13}), (\ref{x15}),  (\ref{x50}),  (\ref{x3})  and (\ref{x4}) we get
\begin{eqnarray}
	\label{136}
	pr\circ	(\Gamma_{5}(\omega)\oplus
	\xi_{5})\circ\mu_{5}\circ\beta_{6}&=&pr\circ(\Gamma_{5}(\omega)\oplus
	\xi_{5})\circ\Big((\beta_6-\sigma_{5}\circ j_5\circ\beta_6)\oplus j_5\circ\beta_6\Big),\nonumber\\
	&=&pr\circ\Gamma_{5}(\omega)\circ(\beta_6-\sigma_{5}\circ d_6)\oplus \xi_5\circ j_5\circ\beta_6,\nonumber\\
	&=&pr\circ\gamma\circ(\beta_6-\sigma_{5}\circ d_6)\oplus \xi_5\circ d_6,\nonumber\\
	&=&\tilde{\gamma}\circ pr\circ(\beta_6-\sigma_{5}\circ d_6)\oplus \xi_6\circ d_6,\nonumber\\
	&=&\tilde{\gamma}\circ\theta_X\oplus \xi_6\circ d_6,\nonumber\\
	pr\circ		\mu_{5}\circ \beta_{6}\circ\xi_6&=&	pr\circ	(\beta_{6}\circ\xi_6- \sigma_{5}\circ j_5\circ\beta_{6}\circ\xi_6)\oplus j_5\circ \beta_{6}\circ \xi_6,\nonumber\\
	&=&	pr\circ(\beta_{6}\circ\xi_6- \sigma_{5}\circ d_{6}\circ\xi_6)\oplus d_{6}\circ\xi_6,\nonumber\\
	&=&\theta_X\circ\xi_6\oplus d_{6}\circ\xi_6.
\end{eqnarray}
Here we use the two formulas $d_{6}=j_{5}\circ \beta_{6}$ and
$\xi_{5 }\circ j_{5}= j_{5} \circ \pi_5(\omega )$ (see \ref{z19}). Now from  (\ref{136}), we get 
\begin{equation}\label{zx1}
	pr\circ\Big(\mu_{5}\circ \beta_{6}\circ\xi_6-		(\Gamma_{5}(\omega)\oplus
\xi_{5})\circ\mu_{5}\circ\beta_{6}\Big)=\theta_X\circ\xi_6-\tilde{\gamma}\circ\theta_X.
\end{equation}
Next, from  (\ref{x6}) and (\ref{zx1}) we obtain 
$$pr\circ\Big(\mu_{5}\circ( \beta_{6}\circ\xi_6-		  \pi_5(\omega)\circ\beta_{6})\Big)=pr\circ\big( \pi_5(\omega) \circ
\sigma_{5} \circ d_{6} - \sigma_{5} \circ \xi_5\circ d_{6}\big)+\theta_X\circ\xi_6-\tilde{\gamma}\circ\theta_X,$$
and using    the formula (\ref{x3}) leads to
$$	pr\circ\Big(\mu_{5}\circ( \beta_{6}\circ\xi_6-		 \pi_5(\omega)\circ\beta_{6})\Big)=pr\circ\big( \pi_5(\omega) \circ
\sigma_{5} \circ j_5\circ\beta_{6} - \sigma_{5} \circ j_5 \circ \pi_5(\omega)\circ\beta_{6}\big)+pr\circ g\circ d_6,$$
which can also be written as
\begin{equation}\label{zx8}
		pr\circ\Big(\mu_{5}\circ( \beta_{6}\circ\xi_6-		 \pi_5(\omega)\circ\beta_{6})\Big)=pr\circ\big( \pi_5(\omega) \circ
\sigma_{5} \circ j_5- \sigma_{5} \circ j_5 \circ \pi_5(\omega)+ g\circ j_5\Big)\circ\beta_6.
\end{equation}
As $$j_5\circ(\beta_{6}\circ\xi_6-  \pi_{5}(\omega)\circ\beta_6)=j_5\circ\beta_{6}\circ\xi_6-  j_5\circ\pi_{5}(\omega)\circ\beta_6=d_{6}\circ\xi_6- \xi_{5 }\circ d_6=0,$$ 
it follows that $\im (\beta_{6}\circ\xi_6-  \pi_{5}(\omega)\circ\beta_6)\subset \Gamma_5(X)$ and applying Remark \ref{rr2}, we get  
$$\mu_{5}\circ( \beta_{6}\circ\xi_6-		  \pi_5(\omega)\circ\beta_{6})=\beta_{6}\circ\xi_6-		  \pi_5(\omega)\circ\beta_{6},$$
therefore, the relation (\ref{zx8}) becomes
$$	pr\circ\Big( \beta_{6}\circ\xi_6-		  \pi_5(\omega)\circ\beta_{6}\Big)=pr\circ\big( \pi_5(\omega) \circ
\sigma_{5} \circ j_5- \sigma_{5} \circ j_5 \circ \pi_5(\omega)+ g\circ j_5\Big)\circ\beta_6,$$
in the other words we get
\begin{equation}\label{zx4}
	pr\circ\Big( \beta_{6}\circ\xi_6-	\big(	  \pi_5(\omega)+ \pi_5(\omega) \circ
\sigma_{5} \circ j_5-\sigma_{5} \circ \xi_{5 }\circ j_{5}+ g\circ j_5\big)\circ\beta_6\Big)=0\in \coker b_6.
\end{equation}
Let us put 
\begin{equation}\label{zx3}
s= \big(\pi_5(\omega) \circ
\sigma_{5} -\sigma_{5} \circ \xi_{5 }+ g\big)\,\,\,\,\,\,\,\,\,\,\,\,\,\text{ and }\,\,\,\,\,\,\,\,\,\,\,t= 	  \pi_5(\omega)+s\circ j_5.
\end{equation}
Notice that  $j_5\circ\pi_5(\omega) \circ
\sigma_{5} -j_5\circ\sigma_{5} \circ \xi_{5 }=\xi_{5 }\circ j_{5}\circ
\sigma_{5} -j_5\circ\sigma_{5} \circ \xi_{5 }=0$. Here we use  $j_5\circ\pi_5(\omega)=\xi_{5 }\circ j_{5}$ and $j_5\circ\sigma_{5}=id$. Therefore, 
\begin{equation}\label{xx4}
\pi_5(\omega) \circ
\sigma_{5} -\sigma_{5} \circ \xi_{5 }+ g\in \H(\ker d_5; \Gamma_5(X)).
\end{equation}
On the one hand,  the relation (\ref{zx4}) becomes $	pr\circ\big( \beta_{6}\circ\xi_6- t\circ\beta_6\big)=0\in \coker b_6$ which implies that  
$$ \im \Big( \beta_{6}\circ\xi_6- t\circ\beta_6\Big)\subset \im \beta_{6},$$
and since $C_{6}X$ is a free abelian group, there exists a homomorphism $\lambda$ making the following diagram commutes
\begin{equation*} \label{x7}
	\begin{picture}(180,80)(40,35)
		\put(45,98){$\vector(0,-1){40}$}
		\put(37,102){$
			C_{6}X$}
		\put(37,45){$C_{6}X\vector(1,0){118}\vector(1,0){2}\,\im \beta_{6}\subset\Gamma_{5}(X)\overset{pr}{\twoheadrightarrow}\coker b_6$} \put(47,75){\scriptsize
			$\lambda$}
		\put(52,98){$\vector(3,-1){130}$} \put(130,75){\scriptsize $\beta_{6}\circ\xi_6- t\circ\beta_6$}
		\put(107,48){\scriptsize $\beta_{6}$}	
	\end{picture}
\end{equation*}

\noindent It is worth mentioning  that the homomorphism $\lambda$ can be chosen such that
\begin{equation}\label{x8}
	\lambda(x)=0 \,\,\,\,\,\,\,\,\,\,\,\,\,,\,\,\,\,\,\,\,\,\,\,\, \forall x\in\ker d_{6}.	
\end{equation}

On the other hand, according to Lemma \ref{l1}, the map $\omega:X^5\to X^5$ and the  homomorphism $\pi_5(\omega) \circ
\sigma_{5} -\sigma_{5} \circ \xi_{5 }$, given in (\ref{xx4})),  allow us to define   a map $\eta:X^5\to X^5$ such that 
\begin{equation}\label{zx5}
\pi_{5}(\eta)=\pi_{5}(\omega)+(\pi_5(\omega) \circ
\sigma_{5} -\sigma_{5} \circ \xi_{5 }+ g)\circ j_5.
\end{equation}
 Consequently, if we set $\rho=\xi_6-\lambda$, then the following diagram commutes

\begin{picture}(300,80)(10,40)
	\put(75,100){ $ C_{6}X
		\hspace{1mm}\vector(1,0){140}\,C_{6}X$}
	\put(89,76){\scriptsize $\beta_{6}$} \put(250,76){\scriptsize $\beta_{6}$}
	\put(87,97){$\vector(0,-1){39}$} \put(248,96){$\vector(0,-1){39}$}
	\put(165,103){\scriptsize $\rho$} \put(160,52){\scriptsize
		$\scriptsize \pi_{5}(\eta )$} \put(73,48){ $ \pi_{5}(X^5)\hspace{1mm}\vector(1,0){130}\hspace{1mm}\pi_{5}(X^5)$} 
\end{picture}

\noindent Indeed, an easy computation shows that
\begin{eqnarray}
	\beta_{6} \circ \rho -\pi_{5}(\eta )\circ\beta_{6}&=& \beta_{6}\circ \xi_6- \beta_{6}\circ\lambda -\pi_{5}(\eta )\circ\beta_{6},\nonumber\\
	&=&  \beta_{6}\circ \xi_6- \beta_{6}\circ\xi_6+ t\circ\beta_6-\pi_{5}(\eta )\circ\beta_{6}.\nonumber
\end{eqnarray}
Taking into account of the two relations (\ref{zx3})  and (\ref{zx5}) we deduce that
\begin{eqnarray}
t\circ\beta_6&=& \big(	  \pi_5(\omega)+ \pi_5(\omega) \circ
\sigma_{5} \circ j_5-\sigma_{5} \circ \xi_{5 }\circ j_{5}+ g\circ j_5\big)\circ\beta_6,\nonumber\\
\pi_{5}(\eta)\circ\beta_6&=&\pi_{5}(\omega)\circ\beta_6+g\circ j_5\circ\beta_6-\sigma_{5} \circ\xi_5\circ j_5\circ\beta_6+\pi_5(\omega) \circ
\sigma_{5}\circ j_5\circ\beta_6.\nonumber
\end{eqnarray}
As a result, we get $\beta_{6} \circ \rho -\pi_{5}(\eta )\circ\beta_{6}=0$. Hence,  applying Remark \ref{r01}, we can  extend the map  $\eta $  to get a map $$\tilde{\eta} :X=X^6 \longrightarrow  X=X^6,$$ with $C_{6}(\tilde{\eta}) = \rho$. Finally, let us prove that 
\begin{equation}\label{x10}
	H_{i}(\tilde{\eta})=f_{i}\,\,\,\,\,\,\,\,\,\,\,\,\,,\,\,\,\,\,\,\,\,\,\,\, i=3,4,5,6.
\end{equation}
First, as  $\tilde{\eta}$ is an extension of  the map $\eta$ to $X$ and from Lemma \ref{l1}, it follows that 
$$C_{\leq 5}(\tilde{\eta})=C_{\leq 5}(\eta)=C_{\leq 5}(\omega)   :C_{\leq 5}(X)\to C_{\leq 5}(X).$$ 
Next, if $x\in\ker d_{6}$, then by (\ref{x8}), we get
$$\rho(x)=\xi_6(x)-\lambda(x)=\xi_6(x).$$
Therefore, the chain map $C_{\leq 6}(\tilde{\eta}) :C_{\leq 5}(X)\to C_{\leq 6}(X)$  satisfies the relation	 
$$C_{\leq 6}(\tilde{\eta})=\xi_{\leq 6},$$
meaning that the relation (\ref{x10}) is satisfied.
\end{proof}

 \begin{corollary}
	\label{c22}                                                    
If $X$ is an   object  in  $\bf{CW^6_2}$, then   the following sequence of groups is exact. 
$$0\to\E_*(X) \to\E(X)\overset{\Psi}{\longrightarrow  } \Gamma\mathcal{S}(X)\to 0.$$
\end{corollary}
\begin{proof}
It follows from  Theorem  \ref{t4} and the following observation
$$\ker \Psi=\Big\{[\alpha]\in\E(X) \mid \Psi([\alpha]) =H_{*}(\alpha)=id\Big\}=\E_*(X).$$	
 Therefore $\E(X)/\E_*(X)\cong\Gamma\mathcal{G}(X)$.
\end{proof}

We end this work  by given two examples illustrating the method developed in this paper.   
 \begin{example}
	\label{e5}Let $X$  be  an object of $\bf{CW^6_2}$ such that 
$$H_{6}(X)=\Z\,\,\,\,\,\,\,\,\,\,\,\,\,,\,\,\,\,\,\,\,\,\,\,\,H_{4}(X)\cong \Z_{p}\,\,\,\,\,\,\,\,\,\,\,\,\,,\,\,\,\,\,\,\,\,\,\,\,H_{3}(X)\cong \Z_{n},$$
with $gcd(p,2)=gcd(n,2)=1$. An easy computation shows that
$$H_{4}(X)\otimes\Z_{2}=0\,\,\,\,\,,\,\,\,\,\,\K_{2}(H_{3}(X);\Z)=0\,\,\,\,\,,\,\,\,\,\,\Gamma_5(X)=\Gamma_4(X)=\Gamma_3(X)=0.$$
As a consequence,  WES$(X)$, given in  Corollary  \ref{c01},  can be written as follows 
$$H_{6}(X)\overset{b_{6}}{\longrightarrow}0\rightarrow 
\pi_{5}(X)\overset{h_{5}}{\longrightarrow}H_{5}(X)\rightarrow  0$$ 
In this case $\Gamma\mathcal{G}(X)$ is  the group of  the following  pairs $$(f_{6},f_5)\in\aut(\Z)\times\aut(H_{5}(X))\cong \Z_2\times\aut (H_{5}(X))$$ for which there exists $\phi$ making the following diagram commutes

	\begin{picture}(300,100)(-25,05)
	\put(67,80){ $\Z\overset{b_{6}}{\longrightarrow  }0\longrightarrow 
		\pi_{5}(X)\overset{}{\longrightarrow} H_{5}(X)\rightarrow  0$}
	\put(135,76){$\vector(0,-1){50}$}
	\put(70,17){$\Z\overset{b_{6}}{\longrightarrow  }0\longrightarrow 
		\pi_{5}(X)\overset{}{\longrightarrow} H_{5}(X)\rightarrow  0$}
	\put(75,76){$\vector(0,-1){50}$} \put(65,50){\scriptsize
		$f^{}_{6}$} 
	\put(190,76){$\vector(0,-1){50}$} 
	\put(192,50){\scriptsize $f^{}_{5}$}
	\put(100,76){$\vector(0,-1){50}$}
	 \put(137,50){\scriptsize $\phi$}
\end{picture}

\noindent  Thus $\E(X)/\E_*(X)\cong\Gamma\mathcal{G}(X)\cong \aut(\Z)\times \aut(H_{5}(X))\cong \Z_2\times \aut(H_{5}(X)).$
Moreover, if $H_{5}(X)\cong \Z_{m}$, then 
$$\E(X)/\E_*(X)\cong \Z_2\times \Z^*_{m}.$$
where $\Z^*_{m}$ is the multiplicative group of the units of $\Z_{m}$.
\end{example}
\begin{example}
	\label{e8}Let $X$  be  an object of $\bf{CW^6_2}$ such that 
	$$H_{6}(X)=\Z\,\,\,\,\,,\,\,\,\,\,H_{4}(X)\cong \Z_{2}\oplus\Z_{2}\,\,\,\,\,,\,\,\,\,\,H_{3}(X)\cong \Z_{n}\,\,\,\,\,,\,\,\,\,\,H_{5}(X)\cong \Z_{8},$$
 where $gcd(n,2)=1$.  Therefore we have 
	$$H_{4}(X)\otimes\Z_{2}=\Z_{2}\oplus\Z_{2}\,\,,\,\,\K_{2}(H_{3}(X);\Z)=0\,\,,\,\,\Gamma_5(X)=\Z_{2}\oplus\Z_{2}\,\,,\,\,\Gamma_4(X)=\Gamma_3(X)=0.$$
 Then   WES$(X)$  can be written as follows 
\begin{equation}\label{x31}
	\Z\overset{b_6}{\longrightarrow}\Z_{2}\oplus\Z_{2}\rightarrow 
\pi_{5}(X)\twoheadrightarrow\Z_{8}. 
\end{equation}
If we choose  $b_6(1)=(1,0)$, then  (\ref{x31}) can be  split to give the following 
$$\Z\overset{b_6}{\longrightarrow}\Z_{2}\oplus\Z_{2}\,\,\,\,\,\,\,\,\,\,\,\,\,,\,\,\,\,\,\,\,\,\,\,\,\,\,\Z_{2}\rightarrowtail
\pi_{5}(X)\twoheadrightarrow\Z_{8}.$$ 
therefore $[\pi_{5}(X)]\in \mbox{Ext}(\Z_{8},  \Z_{2})\cong \Z_{2}$. 
Recall that  $\Gamma\mathcal{G}(X)$ is  the group of  the  automorphisms 
$$(f_{6},f_5)\in\aut(\Z)\times \aut(\Z_{8})\cong \Z_2\times\Z_4\cong\{-1,1\}\times \{1,3,5,7\},$$
such that 
$	(f_5)^*([\pi_{5}(X)] )= (\tilde{\gamma})_*([\pi_{5}(X)]),$ 
where 
$$(f_5)^*,(\tilde{\gamma})_*:\mbox{Ext}(\Z_{8},  \Z_{2})\cong \Z_{2}\to \mbox{Ext}(\Z_{8},  \Z_{2})\cong \Z_{2}\,\,,\,\,\,f_5\in \aut(\Z_{8})\cong\Z_4\cong \{1,3,5,7\}.$$
Here we invoke the formula (\ref{y15}). Note that  $\gamma\in \aut(\Gamma_5(X))=\aut(\Z_{2}\oplus\Z_{2})=\{id\}$ which implies   that $\tilde{\gamma}=id$ and $(\tilde{\gamma})_*=id$. Therefore, 
\begin{equation}\label{y22}
(f_5)^*([\pi_{5}(X)] )= (\tilde{\gamma})_*([\pi_{5}(X)])=[\pi_{5}(X)].
\end{equation}
We distinguish two cases.

\noindent Case 1:  $[\pi_{5}(X)]=0$. In this case the formula (\ref{y22}) is trivially satisfied. Thus, $$\Gamma\mathcal{G}(X)=\Z_2\oplus\Z_4$$

\noindent Case 2:  $[\pi_{5}(X)]=1$. In this case the formula (\ref{y22}) becomes 
$$(f_5)^*(1 )= (\tilde{\gamma})_*([1])=[1],$$
which is also  trivially satisfied for every $f_5\in \aut(\Z_{8}).$  
In this case $\Gamma\mathcal{G}(X)$ is  the group of  the  automorphisms 
$(f_{6},f_5)\in \Z_2\times\Z_4\cong\{-1,1\}\times \{1,3,5,7\}.$ 

Thus, $\E(X)/\E_*(X)\cong\Gamma\mathcal{G}(X)\cong  \Z_2\oplus\Z_4.$
\end{example}





\bibliographystyle{amsalpha}

\end{document}